\documentclass[portuges,12pt,letter]{article}
\usepackage[centertags]{amsmath}
\usepackage{amsfonts}
\usepackage{newlfont}
\usepackage{amscd}
\usepackage{graphics}
\usepackage{epsfig}
\usepackage{indentfirst}
\usepackage{amsxtra}
\usepackage[latin1]{inputenc}
\usepackage{amssymb, amsmath}
\usepackage{amsthm}
\usepackage[mathscr]{eucal}

\newtheorem{thm}{Theorem}[section]

\newtheorem{obe}[thm]{Remark}

\title{\bf{A duality principle for a semi-linear model in micro-magnetism}}

\author{Fabio Silva Botelho }

\begin{document}
\maketitle

\begin{abstract}
This article develops a duality principle for a semi-linear model in micro-magnetism. The results are obtained through standard tools of convex analysis and the
Legendre transform concept.
We emphasize the dual variational formulation presented is concave and suitable for numerical computations. Moreover, sufficient optimality conditions are also established. \\
\end{abstract}

\section{Introduction}
This article develops a dual variational formulation for a semi-linear
model in micro-magnetism. For the primal formulation we refer to
references \cite{120,360} for details. In particular we refer to the original results presented in \cite{360}, emphasizing  the present work is their natural continuation and extension.  We also highlight the present work  develops real relevant improvements relating the previous similar results in \cite{120}.

At this point we start to describe the primal formulation.

Let $\Omega \subset \mathbb{R}^3$ be an open bounded set with a
 a regular (lipschitzian) boundary denoted by
$\partial \Omega$. By a regular lipschitzian  boundary $\partial \Omega$ we mean regularity enough so that the  Sobolev  imbedding theorem and relating results, the trace theorem and the standard Gauss-Green formulas of integration by parts to hold. The corresponding outward normal to $\partial \Omega$ is denoted by $\textbf{n}=(n_1,n_2,n_3)$. Also, we denote by $\mathbf{0}$ either the zero vector in $\mathbb{R}^3$ or the zero in an appropriate function space.

Under such assumptions and notations, consider problem of finding
the magnetization $m:\Omega \rightarrow \mathbb{R}^3$, which minimizes the functional
\begin{eqnarray}
J(m,f)&=&\frac{\alpha}{2}\int_\Omega |\nabla m|^2_2\;dx+\int_\Omega
\varphi(m(x))\;dx-\int_\Omega
H(x)\cdot m \;dx \nonumber \\ &&+\frac{1}{2}\int_{\mathbb{R}^3} |f(x)|^2_2\;dx,
\end{eqnarray}
where
\begin{equation}
m=(m_1,m_2,m_3) \in W^{1,2}(\Omega;\mathbb{R}^3)\equiv Y_1, \; |m(x)|_2 =1, \text{ in } \Omega
\end{equation}
and  $f \in L^2(\mathbb{R}^3;\mathbb{R}^3)\equiv Y_2$ is the
unique field determined by the simplified Maxwell's equations
\begin{equation}
curl (f)=\mathbf{0},\; \; div(-f+m \chi_{\Omega})=0, \text{ in }
\mathbb{R}^3.
\end{equation}
Here $H \in L^2(\Omega;\mathbb{R}^3)$ is a known external field
and $\chi_\Omega$ is a function defined by
 \begin{equation}\label{95}
\chi_\Omega(x)=\left \{
\begin{array}{ll}
1, &\text{ if }\; x \in \Omega,
 \\
 0, & \text{ otherwise.} \end{array} \right.\end{equation}

 The term $$\frac{\alpha}{2}\int_\Omega |\nabla m|^2_2\;dx$$ is called the
exchange energy, where $$|m|_2=\sqrt{\sum_{k=1}^3 m_k^2}$$ and $$|\nabla m|_2^2=\sum_{k=1}^3 |\nabla m_k|^2_2.$$ Finally, $\varphi (m)$ represents the anisotropic
contribution and is given by a multi-well functional whose minima
establish the preferred directions of magnetization.
\begin{obe} Here some brief comments on the references. Relating and similar problems are addressed in \cite{120}. The basic results on convex and variational analysis used in this text may be found in \cite{6,120,29,12}. About the duality principles, we have been greatly inspired and influenced by the work of J.J. Telega and W.R. Bielski. In particular, we would refer to \cite{85}, published in 1985, as the first article to  successfully apply the convex analysis approach to non-convex and non-linear mechanics.

Finally, an extensive study on Sobolev spaces may be found in \cite{1}.
\end{obe}
\begin{obe} At some points of our analysis we refer to the problems in question after discretization. In such a case we referring to their approximations in a
finite element or finite differences context.
\end{obe}

\section{The duality principle for the semi-linear model}
 We consider first the case of a uniaxial material where
$\varphi(m)=\beta (1-|m\cdot e|)$.

Observe that $$ \varphi(m)=\min\{\beta(1+m \cdot e),\beta(1-m \cdot e)\} $$
where $\beta >0$ and $e \in \mathbb{R}^3$ is a unit vector.

The main duality principle is summarized by the following theorem.

\begin{thm} Considering the previous statements and notations, define $J:Y_1 \times Y_2 \times B \rightarrow \overline{R}=\mathbb{R} \cup \{+\infty\}$ by
\begin{eqnarray}
J(m,f,t)&=& G_0(m)-\frac{K}{2}\langle m_i,m_i \rangle_{L^2}+G_1(m,t)+G_2(f)
\nonumber \\ &&+Ind_0(m)+Ind_1(m,f)+Ind_2(f),
\end{eqnarray}
where,
$$G_0(m)= \frac{\alpha}{2} \langle \nabla m_i, \nabla m_i \rangle_{L^2},$$
$$G_1(m,t)= \int_\Omega (t g_1(m)+(1-t) g_2(m))\;dx-\langle H_i,m_i \rangle_{L^2}+\frac{K}{2} \langle m_i,m_i \rangle_{L^2},$$
$$G_2(f)=\frac{1}{2}\int_{\mathbb{R}^3} |f(x)|^2\;dx,$$
$$g_1(m)=\beta (1+m \cdot e),$$
$$g_2(m)=\beta (1-m \cdot e),$$
\begin{equation}Ind_0(m)=\left\{
\begin{array}{lr}
0,& \text{ if } |m(x)|_2=1, \text{ in } \Omega\\
+\infty,& \text{ otherwise,} \end{array}\right.\end{equation}
\begin{equation}Ind_1(m,f)=\left\{
\begin{array}{lr}
0,& \text{ if } div (-f+m \chi_{\Omega})=0, \text{ in } \mathbb{R}^3\\
+\infty,& \text{ otherwise,} \end{array}\right.\end{equation}
\begin{equation}Ind_2(f)=\left\{
\begin{array}{lr}
0,& \text{ if } curl f= \mathbf{0}, \text{ in } \mathbb{R}^3\\
+\infty,&  \text{ otherwise.} \end{array}\right.\end{equation}

We recall the present case refers to a uniaxial material with exchange of energy, that is $\alpha>0$.

Here, $e=(e_1,e_2,e_3) \in \mathbb{R}^3$ is a unit vector.

Under such hypotheses, we have,
$$ \inf_{(m,f,t) \in Y_1 \times Y_2 \times B}\{J(m,f,t)\} \geq \sup_{ \lambda \in A^*}\{\tilde{J}^*(\lambda)\},$$
where
$$\tilde{J}^*(\lambda)=\inf_{ (z^*,t) \in Y^*_4(\lambda) \times B} J^*(\lambda,z^*,t),$$
$$J^*(\lambda,z^*)= \tilde{F}^*(z^*)-\tilde{G}^*(\lambda,z^*,t),$$
and where, for the discretized problem version,
\begin{eqnarray}
\tilde{F}^*(z^*)= \sup_{ m \in Y_1} \left\{ \langle z_i^*, \nabla m_i \rangle_{L^2}+G_0(m)-\frac{K}{2} \langle m_i, m_i \rangle_{L^2}\right\}.
\end{eqnarray}

Here $K>0$ is such that
$$-G_0(m)+\frac{K}{2} \langle m_i,m_i \rangle_{L^2}>0,\; \forall m \in Y_1,\; m \neq \mathbf{0}.$$
\begin{eqnarray}
\tilde{G}^*(\lambda,z^*,t)&=& G_1^*(\lambda,z^*,t)+G_2^*(\lambda)
\nonumber \\ &=&
\sup_{(m,f) \in Y_1 \times Y_2}\{ \langle z_i^*, \nabla m_i \rangle_{L^2}
+ \langle \lambda_2, div(m \chi_\Omega)-f) \rangle_{L^2} \nonumber \\
&& +\langle \lambda_1, curl f \rangle_{L^2} \nonumber \\ &&
-G_1(m,t)-G_0(f)\}
\end{eqnarray}
where, more specifically,
\begin{eqnarray}G_1^*(\lambda,z^*,t)&=&\frac{1}{2} \int_\Omega \frac{ \sum_{i=1}^3 \left(-\frac{\partial \lambda_2}{\partial x_i}+H_i+\beta(1-2t)e_i-div\; z_i^*\right)^2}
{\lambda_3+K}\;dx \nonumber \\ && -\frac{1}{2} \int_\Omega \lambda_3\;dx+\frac{1}{2} \int_\Omega \beta \;dx,\end{eqnarray}

$$
G_2^*(\lambda)=\frac{1}{2}\|\nabla \lambda_2-curl^* \lambda_1\|_2^2.$$

Also,
$$A_1=\{ \lambda \in Y_3\;:\; \lambda_3+K>0, \text{ in } \Omega\},$$
and from the standard second order sufficient optimality condition for a local minimum in $m$, we define
\begin{eqnarray}A_2&=&\{\lambda \in  Y_3\;:\; \nonumber \\ && G_0(m)+\langle \lambda_3, \sum_{i=1}^3 m_i^2 \rangle_{L^2}>0, \nonumber \\ &&
\forall m \in Y_1,\; \text{ such that } m \neq \mathbf{0}\}\end{eqnarray}
where
$$A^*=A_1 \cap A_2,$$
$$Y=Y^*=L^2(\Omega; \mathbb{R}^3)= L^2,$$
$$\lambda=(\lambda_1,\lambda_2,\lambda_3) \in Y_3=W^{1,2}(\mathbb{R}^3;\mathbb{R}^3) \times W^{1,2}(\mathbb{R}^3) \times L^2(\Omega),$$
$$Y_1=W^{1,2}(\Omega;\mathbb{R}^3),$$
$$Y_2=W^{1,2}(\mathbb{R}^3; \mathbb{R}^3),$$
$$Y_4^*(\lambda)=\{z^* \in [Y^*]^3\;:\; z^*_i \cdot \mathbf{n}+\lambda_2 n_i=0, \text{ on } \partial \Omega, \forall i \in \{1,2,3\}\},$$
$$B=\{ t \text{ measurable }\;:\; 0 \leq t \leq 1, \text{ in } \Omega\}.$$

Finally, suppose there exists $(\lambda_0,z_0^*,t_0) \in A^* \times Y^*_4(\lambda_0) \times B$ such that for an appropriate $\lambda_4 \in L^2(\Omega)$ we have
$$\delta \left[J^*(\lambda_0,z_0^*,t_0) +\langle \lambda_4, t_0^2-t_0 \rangle_{L^2}\right]=\mathbf{0},$$
$$\delta^2_{z^*z^*} J^*(\lambda_0,z_0^*,t_0) >\mathbf{0}$$ and  for the concerning Hessian
$$\det\left\{\delta_{z^*,t}^2\left[J^*(\lambda_0,z_0^*,t_0)+\langle \lambda_4, t_0^2-t_0 \rangle_{L^2}\right]\right\}> 0, \text{ in } \Omega.$$

Under such hypotheses,
\begin{eqnarray}
J(m_0,f_0,t_0) &=&\inf_{(m,f,t) \in Y \times Y_1 \times B} J(m,f,t)
\nonumber \\ &=& \sup_{\lambda \in A^* } \tilde{J}^*(\lambda) \nonumber \\ &=& \tilde{J}^*(\lambda_0)
\nonumber \\ &=& J^*(\lambda_0,z_0^*,t_0).
\end{eqnarray}
\end{thm}
\begin{proof}
Observe that
\begin{eqnarray} &&G_1^*(\lambda,z^*,t)+G_2^*(\lambda) \nonumber \\ &\geq&
\langle z_i^*,  \nabla m_i \rangle_{L^2}
+ \langle \lambda_2, div(m \chi_\Omega)-f) \rangle_{L^2} \nonumber \\
&& +\langle \lambda_1, curl f \rangle_{L^2} \nonumber \\ &&+\left\langle \lambda_3, \sum_{i=1}^3 m_i^2-1 \right\rangle_{L^2}
-G_1(m,t)-G_2(f) \nonumber \\ &\geq& \langle z_i^*, \nabla  m_i \rangle_{L^2}
-G_1(m,t)-G_2(f)-Ind_0(m)-Ind_1(m,f)-Ind_2(f),\end{eqnarray}
$\forall (m,f,t) \in Y_1 \times Y_2 \times B, z^* \in Y^*_4(\lambda),\; \lambda \in A^*$
so that,
\begin{eqnarray} &&-\tilde{F}^*(z^*)+G_1^*(\lambda,z^*,t)+G_2^*(\lambda) \nonumber \\
 &\geq&-\tilde{F}^*(z^*)+ \langle z_i^*, \nabla   m_i \rangle_{L^2}
-G_1(m,t)-G_2(f)\nonumber \\ &&-Ind_0(m)-Ind_1(m,f)-Ind_2(f),\end{eqnarray}
$\forall (m,f,t) \in Y_1 \times Y_2 \times B, z^* \in Y^*_4(\lambda),\; \lambda \in A^*$
and hence,
\begin{eqnarray} &&\sup_{z^* \in Y^*_4(\lambda)}\{-\tilde{F}^*(z^*)+G_1^*(\lambda,z^*,t)+G_2^*(\lambda) \}\nonumber \\ &\geq&
 \sup_{ z^* \in Y^*_4(\lambda)}\{-\tilde{F}^*(z^*)+ \langle z_i^*, \nabla m_i \rangle_{L^2}\}
-G_1(m,t)-G_2(f)\nonumber \\ &&-Ind_0(m)-Ind_1(m,f)-Ind_2(f),\end{eqnarray}
$\forall (m,f,t) \in Y_1 \times Y_2 \times B, \; \lambda \in A^*$
so that
that is,
\begin{eqnarray} &&\sup_{z^* \in Y^*_4(\lambda)}\{-\tilde{F}^*(z^*)+G_1^*(\lambda,z^*,t)+G_2^*(\lambda)\} \nonumber \\ &\geq&
 -G_0(m)+\frac{K}{2} \langle m_i,m_i \rangle_{L^2}
-G_1(m,t)-G_2(f)\nonumber \\ &&-Ind_0(m)-Ind_1(m,f)-Ind_2(f)\nonumber \\ &=& -J(m,f,t),\end{eqnarray}
$\forall (m,f,t) \in Y_1 \times Y_2 \times B, \; \lambda \in A^*.$
Thus,
\begin{eqnarray}J(m,f,t) &\geq& \inf_{ z^* \in Y^*_4(\lambda)}\{ \tilde{F}^*(z^*)-G_1^*(\lambda,z^*,t)+G_2^*(\lambda)\}
\nonumber \\ &=& \inf_{z^* \in Y^*_4(\lambda)} J^*(\lambda,z^*,t), \forall (m,f,t) \in Y_1 \times Y_1 \times B, \lambda \in A^*.
\end{eqnarray}
Therefore,
\begin{eqnarray}\label{t1}\inf_{(m,f,t) \in Y_1 \times Y_2 \times B} J(m,f,t) &\geq& \sup_{ \lambda \in A^*}\left\{\inf_{(z^*,t) \in Y^*_4(\lambda)\times B} J^*(\lambda,z^*,t)\right\} \nonumber \\ &=& \sup_{ \lambda \in A^*} \tilde{J}^*(\lambda).
\end{eqnarray}
By the hypotheses, $(\lambda_0,z_0^*,t_0) \in A^* \times Y^*_4(\lambda_0) \times B$ is such that
$$\delta \{J^*(\lambda_0,z_0^*,t_0)+ \langle \lambda_4,t_0^2-t_0 \rangle_{L^2}\}=\mathbf{0}.$$

From the variation in $z^*$,
\begin{equation}\label{t2}\frac{\partial \tilde{F}^*(z_0^*)}{\partial z_i^*}-\nabla (m_0)_i=0, \text{ in } \Omega,\end{equation}
where
\begin{equation}\label{t3}(m_0)_i=\frac{-\frac{\partial (\lambda_0)_2}{\partial x_i}+H_i+\beta(1+2t_0)e_i-div[(z_0^*)_i]}{(\lambda_0)_3+K}.\end{equation}

From the variation in $\lambda_2$,

\begin{equation}\label{t4}div(m_0 \chi_\Omega-f_0)=0, \text{ in } \mathbb{R}^3,\end{equation}
where,
\begin{equation}\label{t5}f_0=curl (\lambda_1)_0-\nabla (\lambda_0)_2.\end{equation}

From the variation in $\lambda_1$, we obtain,
\begin{equation}\label{t6} curl f_0=\mathbf{0}, \text{ in } \mathbb{R}^3.\end{equation}

From (\ref{t2}), we also have,
\begin{equation}\label{t7}\tilde{F}^*(z_0^*)=\langle (z_0^*)_i, \nabla(m_0)_i \rangle_{L^2}+G_0(m_0)-\frac{K}{2}\langle (m_0)_i,(m_0)_i \rangle_{L^2}.\end{equation}
By (\ref{t3}), (\ref{t4}), (\ref{t5}) and (\ref{t6}), we get
\begin{eqnarray}\label{t8} &&G_1^*(\lambda_0,z^*_0,t_0)+G_2^*(\lambda_0) \nonumber \\ &=&
\langle (z_0^*)_i,  \nabla (m_0)_i \rangle_{L^2}
+ \langle \lambda_2, div(m_0 \chi_\Omega)-f_0) \rangle_{L^2} \nonumber \\
&& +\langle (\lambda_0)_1, curl f_0 \rangle_{L^2} \nonumber \\ &&+\left\langle (\lambda_0)_3, \sum_{i=1}^3 (m_0)_i^2-1 \right\rangle_{L^2}
-G_1(m_0,t_0)-G_2(f_0) \nonumber \\ &=& \langle (z_0^*)_i, \nabla  (m_0)_i \rangle_{L^2}
-G_1(m_0,t_0)-G_2(f_0)\nonumber \\ &&-Ind_0(m_0)-Ind_1(m_0,f_0)-Ind_2(f_0),\end{eqnarray}
From (\ref{t7}) and (\ref{t8}), we obtain,
\begin{eqnarray}\label{t9}
&&J^*(\lambda_0,z_0^*,t_0) \nonumber \\ &=&
\tilde{F}^*(z_0^*)-G_1^*(\lambda_0,z_0^*,t_0)-G_2^*(\lambda_0) \nonumber \\ &=&
G_0(m_0)+G_1(m_0,t_0)+G_2(f_0) \nonumber \\ &&+
Ind_0(m_0)+Ind_1(m_0,f_0)+Ind_2(f_0) \nonumber \\ &=& J(m_0,f_0,t_0).
\end{eqnarray}

Finally, from the hypotheses $$\delta^2_{z^*z^*} J^*(\lambda_0,z_0^*,t_0)>\mathbf{0}$$ and $$\det \left\{\delta_{z^*,t}^2\left[J^*(\lambda_0,z_0^*,t_0)+\langle \lambda_4, t_0^2-t_0 \rangle_{L^2}\right]\right\}>0, \text{ in } \Omega.$$
Since the optimization in question in $(z^*,t)$ is quadratic, we obtain,
$$\tilde{J}^*(\lambda_0)=J^*(\lambda_0,z_0^*,t_0).$$

From this, and (\ref{t1}) and (\ref{t9}), we have,
\begin{eqnarray}
J(m_0,f_0,t_0) &=&\inf_{(m,f,t) \in Y \times Y_1 \times B} J(m,f,t)
\nonumber \\ &=& \sup_{\lambda \in A^* } \tilde{J}^*(\lambda) \nonumber \\ &=& \tilde{J}^*(\lambda_0)
\nonumber \\ &=& J^*(\lambda_0,z_0^*,t_0).
\end{eqnarray}
The proof is complete.
\end{proof}


\begin{thebibliography}{}
%
%
\bibitem{1} R.A. Adams and J.F. Fournier, Sobolev Spaces, 2nd edn. Elsevier, New York, 2003.

\bibitem{85} W.R. Bielski and J.J. Telega, A Contribution to Contact Problems for a Class of Solids and Structures,
Arch. Mech., 37, 4-5, pp. 303-320, Warszawa 1985.
\bibitem{360} F. Botelho, Variational Convex Analysis, Ph.D. thesis, Mathematics Department, Virginia Polytechnic Institute and State University, July 2009.
\bibitem{120} F. Botelho, Functional Analysis and Applied Optimization in Banach Spaces,
 Springer Switzerland, 2014.

\bibitem{6} I. Ekeland, R. Temam, Convex Analysis and Variational Problems, North Holland, Amsterdam, 1976.

\bibitem{29} R.T. Rockafellar, Convex Analysis, Princeton University Press, Princeton, 1970.


\bibitem{12} J.F. Toland, {\it A duality principle for non-convex
optimisation and the calculus of variations}, Arch. Rath. Mech.
Anal., {\bf 71}, No. 1 (1979), 41-61.


\end{thebibliography}
\end{document}